\documentclass[12pt]{amsart}
\pdfoutput=1
\usepackage{latexsym, amsmath, amssymb, amsthm, bm}
\usepackage{tikz}
\usetikzlibrary{arrows,shapes,positioning,calc,patterns}
\usepackage[mathscr]{eucal}
\usepackage[colorlinks=true,linkcolor=blue,urlcolor=blue, citecolor=blue]%
  {hyperref}
\usepackage[shortalphabetic]{amsrefs}
\usepackage[T1]{fonten c}
\usepackage{mathptmx}
\usepackage{microtype}
\usepackage[all]{xypic,xy}

\usepackage[centering, includeheadfoot, hmargin=1in, vmargin=1in,
  headheight=30.4pt]{geometry}
\newtheorem{lemma}{Lemma}[section]
\newtheorem{theorem}[lemma]{Theorem}

\newtheorem{proposition}[lemma]{Proposition}

\newtheorem{definition}[lemma]{Definition}

\theoremstyle{definition}
\newtheorem{remark}[lemma]{Remark}

\newtheorem{example}[lemma]{Example}

\renewcommand{\theequation}%
{\arabic{section}.\arabic{lemma}.\arabic{equation}}

\newcommand{\kk}{\ensuremath{\Bbbk}} 

\newcommand{\CC}{\ensuremath{\mathbb{C}}}

\newcommand{\PP}{\ensuremath{\mathbb{P}}} 
 
\newcommand{\RR}{\ensuremath{\mathbb{R}}} 
 
\newcommand{\ZZ}{\ensuremath{\mathbb{Z}}} 
\newcommand{\sE}{\ensuremath{\mathcal{E}}} 
 
\newcommand{\sI}{\ensuremath{\mathcal{I}}}

\newcommand{\sO}{\ensuremath{\mathcal{O}}} 
 
\newcommand{\sP}{\ensuremath{\mathcal{P}}} 

\newcommand{\Osc}{\ensuremath{\mathbf{Osc}}} 
\renewcommand{\geq}{\geqslant}
\renewcommand{\leq}{\leqslant}

\DeclareMathOperator{\conv}{Conv}

\DeclareMathOperator{\Pic}{Pic}
\DeclareMathOperator{\rk}{rk}

\DeclareMathOperator{\Sym}{Sym}
\DeclareMathOperator{\Gr}{Gr}

\usepackage{enumerate}
\usepackage{caption}
\DeclareMathOperator{\im}{Im}
\DeclareMathOperator{\T}{T}
\newcommand\mono{\hookrightarrow}

\newcommand\rat{\dashrightarrow}

\begin{document}

\title[ Higher order Gauss maps]{A note on higher order Gauss maps}

\author[S.~Di~Rocco]{Sandra Di Rocco}
\address{Sandra Di Rocco\\ Department of Mathematics\\ Royal Institute of
  Technology (KTH)\\ 10044 Stockholm\\ Sweden}
\email{\href{mailto:dirocco@kth.se}{dirocco@kth.se}}
\urladdr{\href{http://www.math.kth.se/~dirocco}%
  {www.math.kth.se/~dirocco}}

\author[K.~Jabbusch]{Kelly Jabbusch}
\address{Kelly Jabbusch\\ Department of Mathematics \& Statistics\\ University
  of Michigan--Dearborn \\ 4901 Evergreen Road \\ Dearborn, Michigan
  48128-2406\\ USA}
\email{\href{mailto:jabbusch@umich.edu}{jabbusch@umich.edu}}
\urladdr{\href{http://www-personal.umd.umich.edu/~jabbusch}%
  {http://www-personal.umd.umich.edu/~jabbusch}} 

\author[A.~Lundman]{Anders Lundman}
\address{Anders Lundman\\ Department of Mathematics\\ Royal Institute of
  Technology (KTH)\\ 10044 Stockholm\\ Sweden}
\email{\href{mailto:alundman@kth.se}{alundman@kth.se}}
\urladdr{\href{http://www.math.kth.se/~alundman}%
{www.math.kth.se/~alundman}}


\begin{abstract}
We study Gauss maps of order $k$, associated to a projective variety $X$ embedded in projective space via a line bundle $L.$  We show that if $X$ is a smooth, complete complex variety and $L$ is a  $k$-jet spanned line bundle on $X$, with $k\geq 1,$ then the Gauss map of order $k$ has finite fibers, unless $X=\PP^n$ is embedded by the Veronese embedding of order $k$.  In the case where $X$ is a toric variety, we give a combinatorial description of the Gauss maps of order $k$, its image and the general fibers.  
\end{abstract}

\maketitle
\section{Introduction}

Let $X\subset\PP^N$  be an $n$-dimensional irreducible, nondegenerate projective variety defined over an algebraically closed field $\kk$ of characteristic $0$.  The (classical) Gauss map is the rational morphism
$\gamma :X\dashrightarrow \Gr(n,N)$ that assigns to a smooth point $x$ the projective tangent space of $X$ at $x$, $\gamma(x)={\mathbb T}_{X,x}\cong\PP^n.$ 
It is known that the general fiber of $\gamma$ is a linear subspace of $\PP^N$, and  that the morphism is finite and birational if $X$ is smooth unless $X$ is all of $\PP^N$, \cite{Zak93,KP91, GH79}. 

In \cite{Zak93}, Zak defines a generalization of the  above definition as follows. For $n\leq m\leq N-1$, let $\Gr(m,N)$ be the Grassmanian variety of $m$-dimensional linear subspaces in $\PP^N$,  and define $\sP_m=\overline{\{(x,\alpha)\in X_{sm}\times \Gr(m,N) | {\mathbb T}_{X,x}\subseteq L_\alpha\}},$ where $L_\alpha$ is the linear subspace corresponding to  $\alpha \in \Gr(m,N)$ and the bar denotes the Zariski closure in $X \times \Gr(m,N)$. The {\bf $m$-th Gauss map} is the projection $\gamma_m: \sP_m\to \Gr(m,N)$.  When $m=n$ we recover the classical Gauss map, $\gamma_n=\gamma.$ These generalized Gauss maps still enjoy the property that a general fiber is a linear subspace, \cite[2.3 (c)]{Zak93}. Moreover a general fiber is always finite if $X$ is smooth and $n\leq m\leq N-n+1,$  \cite[2.3 (b)]{Zak93}.

In this paper we consider a different generalization of the Gauss map where, instead of  higher dimensional linear spaces tangent at a point,
we use linear spaces tangent to higher order, namely the osculating spaces. The osculating space of order $k$ of $X$ at a smooth  point $x\in X_{sm},  \Osc_x^k,$ is a linear subspace of $\PP^N$ of  dimension $d_k,$ where $n\leq d_k\leq {n+k\choose n},$ see Definition \ref{def:osc}.  We can then define a rational map
 $\gamma^k:X\dashrightarrow \Gr(d_k-1,N)$ that assigns to a point  $x$ the $k$-th osculating  space of $X$ at $x$, $\gamma^k(x)=\Osc^k_x,$ where $d_k$ is the general $k$-th osculating dimension, see Definition \ref{def:gauss}. Notice that when $k=1$, we recover the classical Gauss map, $\gamma^1=\gamma_n=\gamma.$ We call $\gamma^k$ the {\bf Gauss map of order k}.
This definition was originally introduced in \cite{CA22} and later studied in 
\cite{Pohl} under the name of associated maps. Higher order Gauss maps have subsequently been studied in connection to higher fundamental forms in \cite{L94} and \cite{DI15}, while Gauss maps of order $2$ have been investigated in \cite{FI01}. 

For the classical Gauss map $\gamma$ the linearity of a general fiber is a consequence of the reflexivity property of the projective dual variety of $X.$ Higher order tangency has also been used to generalize the notion of duality and define the higher order dual varieties, $X^k,$ see \cite[Ch 2]{P81}. Unfortunately $X^k$ does not always enjoy reflexivity properties, even if $X$ is nonsingular, as pointed out in \cite[Prop. 1]{P81}. It is therefore reasonable, as  also indicated by the main result in \cite{FI01}, not to expect linearity of the general fiber of $\gamma_k,$ even when the variety is non-singular. Example \ref{non smooth} provides a toric and singular example of a Gauss map with non-linear fibers.

We concentrate instead on establishing a generalization of the finiteness of the Gauss maps of order $k$ when the variety is nonsingular (a property that, as remarked above,  does not always hold for Zak's Gauss maps). First we generalize the classical picture by requiring  the Gauss maps to be regular  when the variety $X$ is nonsingular and thus we consider $k$-jet spanned embeddings, see Definition \ref{def:kjet}, for which $d_k= {n+k\choose n}$ at all points. 
We remark that $k$-jet spanned embeddings  have been extensively studied and classified, see for example \cite{BDRSz, BaSz}. 

The use of certain techniques from projective geometry imposes the assumption of $\kk=\CC.$ We sternly think though that the results in this paper should be extendable to any field of characteristic zero. Theorem \ref{finite} shows:

 \begin{theorem} Let $i: X\hookrightarrow \PP^N$ be a $k$-jet spanned embedding of a nonsingular complex variety. Then 
the Gauss maps $\gamma^s$ are finite for all $s\leq k$, unless $X=\PP^n$ is  embedded by the Veronese embedding of order $k.$ \end{theorem}

Section \ref{sec:toric} is dedicated to giving a combinatorial description of the maps $\gamma^k$ and the images $\gamma^k(X),$ called the $k$-th osculating variety, in the case when $X$ is a toric variety.
In \cite[Theorem 1.1]{FI14} it is shown that if $X_A$ is a toric variety (not necessarily smooth) given by a finite lattice set $A$, then the tangential variety $\overline{\gamma(X)}$ is projectively equivalent to a toric variety $X_B$ where $B$ is obtained by taking appropriate sums of elements in $A$. Theorem~\ref{thm:Bklattice} is a direct generalization of this result and the ideas in the proof. 

\begin{theorem} If $X_A$ is a toric variety given by a set of lattice points $A$ and the embedding is generically $k$-jet spanned, then
there exists a finite set of lattice points $B_k$ and a lattice projection $\pi$ such that $\overline{\gamma^k(X)}$ is projectively equivalent to $X_{B_k}$ and the closure of the irreducible components of the fiber of $\gamma_k$ are projectively equivalent to $X_{\pi(A)}.$
\end{theorem}
%
This description allows us to reprove our finiteness result in the toric setting using combinatorial methods. A simple consequence of the combinatorial proof of finiteness is that the Gauss map of order $k\geq 1$ is birational for smooth $k$-jet spanned toric embeddings, when finite. 

This is false outside the toric category, the Gauss maps of order $k\geq 2$ need not in general be birational when finite, see \cite{FI01}.

We remark that the assumption of $k$-spannedness cannot be relaxed in general. In particular Example~\ref{genkspanned} is a generically $2$-jet spanned smooth surface with positive dimensional fibers under the Gauss map of order $2$.

Theorem~\ref{thm:Bklattice} also makes it possible to compute the image and general fiber of the Gauss map of order $k$ in the toric setting. This is implemented in the Package \textsf{LatticePolytopes}, \cite{LP} for \textsf{Macaulay2}, \cite{M2}.

\subsection*{Future directions and applications} 
The Gauss map is used to define the (normalized) Nash Blow up of a variety. The Gauss map of higher order could lead to a definition
of ``Higher order Nash blow ups". For toric varieties interesting results have been proved for example in \cite{A11}. It is reasonable to expect that a generalization would lead to new resolution properties, at least in the toric category.
Moreover it is worth mentioning that Gauss maps play an important role in characterizing the boundaries of amoebas and in connection with real $A$-discriminants, see for example \cite{K91}. A generalization, using higher order Gauss maps, could yield interesting applications within real algebraic geometry.  We plan to investigate both  directions mentioned above in future work.

\subsection*{Conventions} 
 We work over the field of complex numbers $\CC.$ 
Throughout the paper, $X$ denotes a smooth, complete complex algebraic variety of dimension $n.$ We use additive notation for the group operation in ${\rm Pic}(X).$

\subsection*{Acknowledgements}
The first and third author were partially supported by the VR grants [NT:2010-5563, NT:2014-4763].
The second author was partially supported by the G\"oran Gustafsson foundation. We are thankful to the referees for useful suggestions.
\section{Definitions and background}
\subsection{Restrictions imposed by ample divisors.} In this section we collect the necessary background on  the invariants of ample Cartier divisors used in the proof of the main result. 

Let $L$ be an ample line bundle on $X.$ The {\it nef-value} of $L$ is defined as 
$$\tau(L)=min_\RR\{t\,|\, K_X+tL \text{ is nef }\}.$$
Kawamata's Rationality Theorem  shows that $\tau(L)$ is in fact a rational number. Let $X\to \PP^M$ be the morphism defined by the global sections of an appropriate multiple of $K_X+\tau L$ and let 
$\psi\circ\phi_\tau$ be its Remmert-Stein factorization. The map $\phi_\tau:X\to Y$ has connected fibers and it is called the {\it nef-value morphism}. See \cite[1.5]{BS95} for more details.

An easy way to compute $\tau(L)$ is provided by the following lemma.
\begin{lemma}\label{nefnotample}\cite[1.5.5]{BS95}
Let $L$ be an ample line bundle on $X$ and $\tau\in\RR.$  Then $\tau=\tau(L)$ if and only if
$K_X+\tau L$ is nef but not ample.
\end{lemma}

The nef-value morphism always contracts curves on the variety $X$ (since the defining line bundle is not ample). To state this more precisely, let $\overline{NE(X)}$ be the closure of the  cone generated by the effective $1$-cycles on $X.$

\begin{lemma}\label{face}\cite[4.2.13 (1)]{BS95}
Let $L$ be an ample line bundle on $X.$ Then the nef-value morphism $\phi_\tau$ is the contraction of an extremal face $F_H$ of  $\overline{NE(X)},$ where $H=K_X+\tau L$ and $F_H=H^\perp\cap(\overline{NE(X)}\setminus \{0\}).$
\end{lemma}

Finally we recall a useful classification of Fano varieties, 
based on the length of extremal rays. 
Let $R\in\overline{NE(X)}$ be an extremal ray. The length of $R$ is defined as $l(R)=min\{-K_X\cdot C\,|\, C\text{ is a  rational curve and }[C]\in R\}.$ The cone theorem implies that $0< l(R)\leq n+1.$

\begin{proposition}\cite[6.3.12,]{BS95},\cite{CM02}\label{characterization}
Let $C$ be an extremal rational curve on $X.$ If $-K_X\cdot C=n+1$, then $-K_X$ is ample and $\Pic(X)\cong\ZZ.$
\end{proposition}

\subsection{Osculating spaces.}
Let $L$ be a line bundle on $X$ and let $V=H^0(X,L).$  The coherent sheaf $J_k(L)={p_1}_*(p_2^*(L)\otimes{\mathcal O}_{X\times X}/\sI_{\Delta}^{k+1}
),$ where $\Delta\subset X\times X$ is the diagonal and $p_i$ are the projection maps $p_i: X\times X\to X,$ is locally free of rank ${n+k\choose n}$, and is called the \emph{$k$-jet bundle of $L$}. The fiber at a point $x\in X$ can be identified with $J_k(L)_x\cong H^0(X, L\otimes \sO_X/\mathfrak m_x^{k+1}),$ where $\mathfrak m_x$ is the maximal ideal at $x.$ The quotient map $$j_{k,x}: V\to H^0(X, L\otimes \sO_X/\mathfrak m_x^{k+1})$$ evaluating  a global section and its derivatives of order at most $k$ at the point $x:$ $$j_{k,x}(s)=(s(x),\ldots, \frac{\partial^t s }{\partial {\mathbf x}^t}(x),\ldots)_{1\leq t \leq k}$$ for a coordinate system ${\mathbf x}=(x_1,\ldots,x_n),$ extends to a vector bundle map: $j_k: V\otimes\sO_X\to J_k(L).$
We denote by $U_k\subset X$  the open locus where the vector bundle map $j_k$ obtains its maximal rank $d_k\leq {n+k\choose k}.$ Moreover if $s_0,\dots,s_m$ is a basis for $V$ then the rank of the map $j_k$ at a point $x\in X$ equals the rank of \emph{the matrix of $k$-jets} which is defined as $[J_{k,x}]=[j_k(s_0)|\ldots | j_k(s_m)]$. Notice that the vector space map $j_{k,x}$ induces and inclusion of projective spaces $\PP(j_{k,x}(V))\hookrightarrow \PP(V).$

\begin{definition}\label{def:osc}
The projectivization of the image $\PP(j_{k,x}(V))=\Osc_x^k\subseteq\PP(V)$ is called the {\bf $k$-th osculating space} at $x.$
The integer $d_k$ is the {\bf general osculating dimension }of $L$ on $X.$ If $U_k=X$ the integer $d_k$ is called the  {\bf $k$-th osculating dimension }of $L$ on $X.$ 
\end{definition}
The line bundles for which the $k$-th osculating dimension is maximal define embeddings with high geometrical constraints. These are the embeddings that we will consider in the remainder of the article.
\begin{definition}\label{def:kjet}
Let $L$ be a line bundle on $X$ and let $d_k$ be its general $k$-th osculating dimension.  If $d_k={n+k\choose k}$ then the (rational) map defined by the global sections of $L$ is said to be {\bf generically $k$-jet spanned}. If $d_k={n+k\choose k}$ is the $k$-th osculating dimension then the map is said to be  {\bf  $k$-jet spanned}. 
  \end{definition}

\begin{remark}
Observe that $0$-jet spanned is equivalent to being globally generated. 
Moreover, if a line bundle $L$ is $k$-jet spanned then it is $s$-jet spanned for all $s\leq k$. 
\end{remark}
\begin{remark}
We note here that if $X$ is a smooth algebraic variety and $L$ is a very-ample line bundle then $L^N$ is  k-jet spanned for all $N\geq k,$ see
\cite{BeSo}. Bounds on $N$ for when a multiple $L^N$ of (just) an ample  line bundle $L$ is $k$-jet spanned have been investigated for many classes of varieties, as K3 surfaces \cite{BDRSz} or Abelian varieties \cite{BaSz}. More generally $k$-jet spanned embeddings have been studied for many classes of varieties, see for example \cite{BeSo, Te, BeDRSo}.
\end{remark}

\begin{example}
As a first example, we see that $(X,L)=(\PP^n,\sO_{\PP^n}(k))$ is $k$-jet spanned.  Indeed, a basis of global sections of $\sO_{\PP^n}(k)$ is given by all the degree $k$ monomials in $x_0, \ldots, x_n$.  Thus, the maximal rank of $j_{k,x}$, at any $x \in \PP^n$, is $d_k ={n+ k\choose k}$.  Note that $\sO_{\PP^n}(k)$ is not $l$-jet spanned for $l>k$.  
\end{example}

In the next example we distinguish between $k$-jet spanned and generically $k$-jet spanned.  
\begin{example}
Let $p:X\to\PP^2$ be the blow up of $\PP^2$ at three non-collinear points $p_1, p_2$ and $p_3$ and let $L=-K_X=p^*(\sO_{\PP^2}(3))-E_1-E_2-E_3,$ where the $E_i$ are the exceptional divisors.  Let $l_{ij}$ be the lines in $\PP^2$ connecting $p_i$ and $p_j$ for $1 \leq i < j \leq 3$, and denote by $\tilde{l}_{ij}$ the proper transform.  If $x \in X$ is a point that is not in any exceptional divisor nor in any $\tilde{l}_{ij}$, then the rank of $j_{2,x}$ is $6$; if $x$ lies on the intersection of an exceptional divisor and $\tilde{l}_{ij}$, then the rank of $j_{2,x}$ is $4$; and for any other $x \in X$ the rank of $j_{2,x}$ is $5$ \cite[Theorem 2.1]{LM01}.  Thus for $x$ not in any exceptional divisor nor in any $\tilde{l}_{ij}$, $L$ is $2$-jet spanned at $x$, and hence the embedding defined by $L$ is generically $2$-jet spanned.  However, if $x \in E_i$ or $x \in \tilde{l}_{ij}$, then $L$ is not $2$-jet spanned at $x$, and thus the embedding defined by $L$ is not $2$-jet spanned.  
\end{example}

The generation of $k$-jets imposes strong conditions on intersections with irreducible curves on $X.$
\begin{lemma}\label{restr}
Let $L$ be a $k$-jet spanned line bundle on $X$ and let $C\subset X$ be an irreducible curve. Then
\begin{enumerate}
\item $L\cdot C\geq k;$
\item $L\cdot C=k$ if and only if $C\cong\PP^1$ and $L$ is not $(k+1)$-jet spanned.
\end{enumerate}
\end{lemma}
\begin{proof} Since $L$ is $k$-jet spanned, its restriction to $C$, $L|_C$, is a $k$-jet spanned line bundle on $C.$ Assume now that $L\cdot C\leq k$ so that for any $x \in C$ it holds that  $H^0(C, L|_C\otimes \mathfrak{m}_x^{k+1})=0.$
Because the map $j_{k,x}:H^0(C,L|_C) \to H^0(C, L|_C \otimes \sO/\mathfrak{m}_x^{k+1})$ is surjective  for all points $x \in C$,  we have that $\dim(H^0(C,L|_C))=k+1+\dim(H^0(C, L|_C\otimes \mathfrak{m}_x^{k+1}))=k+1.$ This in turn implies that the map $H^0(C,L|_C)\times C\to J_k(L|_C)$ is an isomorphism, and thus $ J_k(L|_C)=\sO_{C}^{\oplus k+1}.$ But the only smooth curve, $C$,  having  a line bundle, $H$,  with trivial jet bundle is $(C,H) = (\PP^1,\sO_{\PP^1}(k))$, see \cite{FKPT85, DRS01}.  \end{proof}

\subsection{Toric Geometry.}\label{sec:toricbackground}

In this section we provide a short background on relevant parts of toric geometry. References are \cite{Fulton} and \cite{CLS}. 

Let $M$ be a lattice of rank $n$, then the maximum spectrum of the group ring $\CC[M]=\bigoplus_{\mathbf{u}\in M} \CC \mathbf{x}^\mathbf{u}$ is an algebraic torus $T_M=\operatorname{Spec}(\CC[M])\cong(\CC^*)^n$. Moreover a finite subset $A=\{\mathbf{u_0},\ldots,\mathbf{u_N}\}\subseteq M$ induces the following map
\begin{align}
\phi:T_M\cong (\CC^*)^n&\to \PP^{N}\label{torusemb}\\
\mathbf{x}&\mapsto(\mathbf{x}^\mathbf{u_0},\ldots ,\mathbf{x}^\mathbf{u_N})\notag
\end{align}
where $\mathbf{x}=(x_1,\ldots,x_n)$, $\mathbf{u}_i=(u_i^1,\ldots,u_i^n)$ and $\mathbf{x}^{\mathbf{u}_i}=x_1^{u_i^1}\cdots x_n^{u_i^n}$. It is a standard fact that the image, $\im(\phi_A)$, is an algebraic torus $T_{\langle A-A \rangle}$, where $\langle A-A \rangle=\{u-u'\in M \mid u,u'\in A\}$. The closure of the image is a toric variety, $X_A=\overline{\im(\phi_A)}$, which has $T_{\langle A-A \rangle}$ as an open dense subset. 

Recall that  if  $L$ is a line bundle on a normal toric variety $X$ and  $P_L\subset M_\RR$ is the associated polytope then
\begin{equation}\label{eqn:gsectpoly}
H^0(X,L)\cong\bigoplus_{m\in M\cap P_L}\CC \langle x^m\rangle
\end{equation} 
where $m=(m_1,m_2,\dots,m_n)$, $x=(x_1,\dots,x_n)$ and $x^m=x_1^{m_1}x_2^{m_2}\cdots x_n^{m_n},$ after a choice of basis vectors for $M$.
It follows that  the space of global sections of a line bundle on a normal toric variety has a monomial basis. As a consequence the matrix of $k$-jets, $[J_{k,x}]$ has a particularly simple form which makes the toric setting appealing from a computational perspective. In particular, it can be shown that a line bundle $L$ is $k$-jet spanned at the general point of a toric variety $X$ if and only if $L$ is $k$-jet spanned at the image of the point $(1,\ldots,1)$ under the map \eqref{torusemb}, see \cite[p. 3]{Perkinson}.
\section{Higher order Gauss maps}
Let $\dim(H^0(X,L))=\dim(V)=N+1$ and let $Gr(t,N)$ denote the Grassmanian variety of linear spaces $\PP^t\subset \PP(V).$  Assume that $L$ is very ample and thus $X\subset\PP(V),$ which in particular implies that the general $k$-th osculating dimension is $d_k\geq n+1,$ for $k\geq 1.$
\begin{definition}\label{def:gauss}
The Gauss map of order $k$ is the (rational) map:
$$\gamma^k: X\dashrightarrow \Gr(d_k-1,N)$$
assigning to $x\in U_k\subseteq X$ the $k$-th osculating space $\gamma^k(x)=\Osc_{k,x}\cong\PP^{d_k-1}.$

We call the image variety, $\gamma^k(X)$, the {\bf osculating variety of order k}.
\end{definition}
\begin{remark} If $k=1$  then $\Osc_{1,x}={\mathbb T}_{X,x}\cong\PP^n.$
It follows that $\gamma^1=\gamma$ is the classical Gauss map.
\end{remark}
\begin{example} On $X=\PP^n$, $L=\sO_{\PP^n}(k)$ can be considered an extreme case. The line bundle is $k$-jet spanned and thus $d_k={n+k\choose k}$ is the osculating dimension at every point. The line bundle defines the $k$-th Veronese embedding $\PP^n\hookrightarrow \PP^{{n+k\choose k}-1}=\PP(V)$ and the
osculating space at every $x$ is the whole $\PP(V).$ The Gauss map of order $k$ is a regular map contracting the whole $\PP^n$ to a point.
\end{example}

In order to generalize the classical result on the finiteness of the fibers of the Gauss map, we  will now assume that the very ample line bundle $L$ is $k$-jet spanned.    
Then the Gauss map of order $k$ is  a regular map $\gamma^k:X\to \Gr({n+k\choose k}-1,N).$
Consider the so called $k$-jet sequence:
$$0\to \Sym^k(\Omega^1_X) \otimes L \to J_k(L)\to J_{k-1}(L)\to 0$$
An induction argument shows that 
\begin{equation}\label{jetdet}\det(J_k(L))=\frac{1}{n+1}{n+k\choose k}(kK_X+(n+1)L).\end{equation}
In particular, $\det(J_k(L))$ will be ample, nef or globally generated if $kK_X+(n+1)L$ is ample, nef or globally generated, respectively.

The following two lemmas are the key observations for the proof of our main result.
\begin{lemma}\label{det}
Assume $L$ is a  $k$-jet spanned line bundle on $X$ such that  $\det(J_k(L))$ is ample. Then the regular map  $\gamma^k$ is finite. 
\end{lemma}
\begin{proof} Because the line bundle $L$ is $k$-jet spanned on the whole variety $X$ then the $k$-th osculating dimension is $d_k={n+k\choose k}$ and   the map 
$\gamma^k$ is regular. Consider the composition  of the Gauss map of order $k$ with the Pl\"ucker embedding, $pl$:
\[
pl\circ\gamma^k:
\xymatrix{
X\ar@{->}[r]
&\Gr(d_k-1,M) \ar @{^{(}->}[r]
& \PP^T}
\]
Recall that the vector bundle map $j_k: V\otimes\sO_X\to J_k(L)$ is onto and thus $J_k(L)$ is generated by the global sections of $L.$
Because $\gamma^k(x)=\PP(J_k(L)_x)$ for every point  $x\in X$ the composition  $pl\circ\gamma^k$
is the map defined by the global  sections of the   line bundle $\wedge^{{n+k\choose k}} J_k(L)=\det(J_k(L)).$
If this composition has a  fiber $F$ of positive dimension $s\geq 1$  then $\det(J_k(L))^{n-s}\cdot F=0.$ This cannot happen if  $\det(J_k(L))$
is ample. 
\end{proof}

\begin{lemma}\label{trivaldet} Let $\sE$ be a globally generated rank $r$ vector bundle on $X$ such that $\det \sE = \sO_X$, then $\sE \cong\sO_X^{\oplus r}$. \end{lemma}
\begin{proof}
Let $x \in X$ be a point and let $s_1, \ldots, s_r$ be sections generating $\sE$ at $x$.  Then the locus where these $r$ sections do not generate $\sE$ is a divisor in the linear system of $\det \sE$.  Since $\det \sE$ is trivial, the $r$ sections must generate $\sE$ everywhere, and hence $\sE \cong\sO_X^{\oplus r}$. 
\end{proof}

\begin{theorem}\label{finite}
Let $L$ be a $k$-jet spanned line bundle on $X,$ with $k\geq1.$ Then the  Gauss map of order $k,$ $\gamma^k:X\to \Gr\left({n+k\choose k}-1,N\right),$ is finite unless $(X,L)=(\PP^n,\sO_{\PP^n}(k)).$
\end{theorem}
\begin{proof} By Lemma \ref{det} it suffices to prove that $\det(J_k(L))$ is ample unless $(X,L)=(\PP^n,\sO_{\PP^n}(k)).$ In view of formula (\ref{jetdet}) it is in turn sufficient to show that the line bundle $kK_X+(n+1)L$ is ample unless $(X,L)=(\PP^n,\sO_{\PP^n}(k)).$
Assume that $kK_X+(n+1)L$ is not ample.  As previously observed   the vector bundle map $j_k: V\otimes\sO_X\to J_k(L)$ is onto and thus $J_k(L)$ is generated by the global sections of $L$, implying that  $\det(J_k(L))$ is also globally generated and thus nef.
Again formula (\ref{jetdet}) gives that the line bundle 
$kK_X+(n+1)L$ is also nef. By Lemma \ref{nefnotample} we can conclude that the nef-value of $L$ is $\tau(L)=\frac{n+1}{k}.$
Let  $R$ be an extremal ray in the face contracted by the nef-value morphism, as in Lemma~\ref{face}.  Note that we can choose an extremal rational curve $C$, with $[C] \in R$ and $-K_X \cdot C = l(R)$, see \cite[4.2.5]{BS95}.   Then $(kK_X+(n+1)L)\cdot C=0$ and $-C\cdot K_X \leq n+1$.  
But because $L\cdot C\geq k,$  by Lemma \ref{restr}(a), we must have $K_X\cdot C=-n-1$ and $L\cdot C=k.$ Proposition \ref{characterization} implies then that $X$ is a Fano variety, i.e. $-K_X$ is ample, and $\Pic(X)=\ZZ.$
Thus $kK_X+(n+1)L \cong \sO_X$, and so $\det (J_k(L))$ is also trivial.  Since $J_k(L)$ is a globally generated vector bundle with trivial determinant, we can apply Lemma \ref{trivaldet} and conclude that $J_k(L) \cong \sO_X ^{\oplus{n+k\choose k}}$.

By \cite{DRS01}, if $J_k(L)$ is trivial then either $(X,L) = (\PP^n, \sO_{\PP^n}(k))$ or $X$ is an abelian variety and $L$ is trivial.  The second case cannot occur in our situation because $L$ is ample.  Thus we conclude that   $(X,L)=(\PP^n, \sO_{\PP^n}(k))$.  
\end{proof}

The following example shows that, as in the classical case, finiteness cannot be expected if we drop the smoothness assumption.
\begin{example}\label{non smooth}
Consider the toric variety $X$ together with a very ample line bundle $L$ given as the closure in $\PP^5$  of the following torus embedding:
\begin{align*}
\phi:(\CC^*)^2&\to \PP^5\\
 (x,y)&\mapsto (1:x:y:xy:x^2:xy^2)
\end{align*} 
It is readily checked that $(X,L)$ is generically $2$-jet spanned by computing the rank of the matrix of $2$-jets at a general point. However, $\dim(H^0(X,L))=6$ so the Gauss map of order $2$ is a map from $X$ to the one point space $\Gr(5,5)$. Thus $\gamma^2$ must be the contraction of $X$ to a point, and have $X$ as its fiber, in particular $\gamma^2$ is not generically finite. We remark that a direct computation shows that $(X,L)$ is not $2$-jet spanned at the point $(1:0:0:0:0)$, in particular $(X,L)$ is not $2$-jet spanned.
\end{example}

In Section~\ref{sec:toric} we will give a further family of examples: for every pair of integers $n\ge 2$ and $N\ge 2$, we will construct a singular toric variety of dimension $n$ in $\PP^{\binom{n+2}{2}+N-2}$ which is generically $2$-jet spanned, but which has a Gauss map of order $2$ with positive dimensional fibers.  See Example~\ref{ex:infinitefibers}.

Furthermore, as the following example shows, smoothness and only generically $k$-jet spannedness does not in general imply that the general fiber of $\gamma^k$ is finite. 

\begin{example}\label{genkspanned}
Let $X$ be the Del Pezzo surface of degree $5,$ given by the blow up of $\PP^2$ in $4$ points in general position embedded by  the anticanonical bundle $-K_X.$    In \cite[Theorem 2.1]{LM01} it is shown that  $-K_X$ is $2$-jet spanned at all points outside the $4$ exceptional divisors $E_1,\ldots, E_4.$ It follows that the Gauss map of second order is the rational map
$\gamma^2:X\dashrightarrow \Gr(5,5)=\text{pt},$ contracting  $X\setminus (E_1\cup\ldots\cup E_4)$ to a point and 
is in particular   not generically finite.
\end{example}
\section{Toric Gauss maps}\label{sec:toric}
In \cite{FI14} Furukawa and Ito gave combinatorial descriptions of the image and fiber of the classical Gauss map in the toric setting. In this section we will use the techniques introduced in \cite{FI14}  to extend their results to Gauss maps of higher order. Let $M$ be a lattice and let $A=\{u_0,\dots,u_N\}\subset M$. Then as explained in Section~\ref{sec:toricbackground}, $A$ determines a map $\phi_A:T_M\mono\PP^N$ and a toric variety $X_A=\overline{\im(\phi_A)}$. We make the following definitions.

\begin{definition}
Let $A\subset M$ be a finite set of lattice points. $A$ is called \emph{generically $k$-jet spanned} if the associated line bundle $\phi_A^*({\mathcal O}_{\PP^{|A|-1}}(1))$ determines an embedding that is generically $k$-jet spanned.
\end{definition}

\begin{definition}
Assume that $A=\{u_0,\ldots,u_N\}$ is generically $k$-jet spanned and let $d:=d_k=\binom{n+k}{k}$. For every subset $\{u_{i_1},\ldots,u_{i_{d}}\}$ of ${d}$ lattice points in $A$ we denote by $\left[J_{k,(1,\ldots,1)}^{\{u_{i_1},\ldots,u_{i_{d}}\}}\right]$ the matrix of $k$-jets of the torus embedding given by $\phi_{\{u_{i_1},\ldots,u_{i_{d}}\}}$ evaluated at the point $(1,\ldots,1)$. We define the following subset of the lattice $M$:
\[
B_k=\{u_{i_1}+u_{i_2}+\dots+ u_{i_{d}}\mid u_{i_1},\ldots, u_{i_{d}}\in A\text{ and } \det\left[J_{k,(1,\ldots,1)}^{\{u_{i_1},\ldots, u_{i_{d}}\}}\right]\ne 0\}.
\]
\end{definition}

Observe that the assumption $\det\left[J_{k,(1,\ldots,1)}^{\{u_{i_1},\ldots, u_{i_{d}}\}}\right]\ne 0$ is equivalent to saying that the set of lattice points $\{u_{i_1},\ldots,u_{i_{d}}\}$ is generically $k$-jet spanned. In particular  for $k=1$, it holds that $\det\left[J_{k,(1,\ldots,1)}^{\{u_{i_1},\ldots, u_{i_{d}}\}}\right]\ne 0$ if and only if $u_{i_1},\ldots,u_{i_{n+1}}$ span the affine space generated by $A$. The set $B_1$ defined above is denoted by $B$ in \cite{FI14}. Going through the proof of \cite[Theorem 1.1]{FI14} and replacing $B$ with $B_k$ yields the following result.

\begin{theorem}\label{thm:Bklattice}
Let $\pi_k:M\to M'=M/(\langle B_k-B_k\rangle)_\RR\cap M)$ be the natural projection and assume that $(X_A,L_A)$ is generically $k$-jet spanned. The following holds:
\begin{enumerate}[(i)]
\item{The closure $\overline{\gamma^k(X_A)}$ of the Gauss map of order $k$ is projectively equivalent to $X_{B_k}$.}\label{thma}
\item{The image of the restriction of $X_A \rat \overline{\gamma^k(X_A)}$ to $T_M$ is projectively equivalent to the image of the morphism $\T_M\twoheadrightarrow \T_{\langle B_k-B_k\rangle}$ induced by the inclusion $\langle B_k-B_k \rangle\hookrightarrow M$.}\label{thmb}
\item{Let $F$ be an irreducible component of a general fiber of $\gamma^k|_{\T_M}$ with the reduced structure. Then $F$ is a translation of $\T_{M'}$ by an element of $\T_M$. Moreover the closure $\overline{F}$ is projectively equivalent to $X_{\pi(A)}$. In particular the dimension of the general fiber is \[
\delta_\gamma^k(X_A)=\rk M'=n-\rk\langle B_k-B_k\rangle.\]}\label{thmc}
\end{enumerate}
\end{theorem}

Recall that two varieties $X_1\subseteq  \PP^{N_1}$ and $X_2\subseteq \PP^{N_2}$ are said to be \emph{projectively equivalent} if there exist embeddings $j_i: \PP^{N_i}\mono \PP^N$ such that $j_1(X_1)=j_2(X_2)$ and $j_i^*(\sO_{\PP^N}(1))=\sO_{\PP^{N_i}}(1)$.
\begin{proof}
Following the proof of Theorem 1.1 in \cite{FI14} one shows there is a commutative diagram of the following form
\[
\xymatrix{
T_M \ar@{^{(}->}[r] ^{\phi_A} \ar@{_{(}->}@/_0.75pc/[drrr]_{\phi_{B_k}}
&X_A \ar@{-->}[r]^{\gamma^k}
&\Gr(q,N) \ar@{^{(}->}[r]^{pl}
&\PP\left(\bigwedge^{d}V\right)\\
&
&
&\PP^{|B_k|-1} \ar@{_{(}->}[u]_j}
\]
where $pl$ is the Pl\"{u}cker embedding and $j$ is a linear embedding. Here the morphism $j$ is the morphism making the diagram commutative. To describe $j$ more explicitly, consider a subset $u_{i_1},\ldots,u_{i_{d}}$ of $A$. If $u_{i_1},\ldots,u_{i_{d}}$ is  generically $k$-jet spanned then there is a corresponding coordinate $y_{u_{i_1},\ldots,u_{i_{d}}}=x^{\sum_l u_{i_l}}$ in $\PP^{|B_k|-1}$. Moreover the subset $u_{i_1},\ldots u_{i_{d}}$ corresponds to the Pl\"{u}cker coordinate $p_{u_{i_1},\ldots,u_{i_{d}}}$. In these coordinates we define $j\colon \PP^{|B_k|-1}\to \PP\left(\bigwedge^{d}V \right)$ as 
\[
y_{u_{i_1},\ldots,u_{i_{d}}}=x^{\sum_l u_{i_l}}\mapsto (0,\ldots,\underbrace{\det\left[J_{k,(1,\ldots,1)}^{\{u_{i_1},\ldots, u_{i_{d}}\}}\right]x^{\sum_l u_{i_l}}}_{p_{u_{i_1},\ldots,u_{i_{d}}}},0,\ldots,0).
\]
For further details on this construction we refer to \cite{FI14}. By the above diagram 
\[
\overline{\gamma^k(X_A)}=\overline{pl\circ \gamma^k\circ\phi_A(T_M)}=\overline{j\circ \phi_{B_k}(T_M)}=j(X_{B_k})
\]
This proves part \emph{(\ref{thma})}. 

Restricting the morphism $X_A \rat \overline{\gamma^k(X_A)}$ to $T_M$ corresponds to considering the composition $pl\circ \gamma^k \circ \phi_A: T_M\to \PP\left(\bigwedge^{{d}}V\right)$. As $T_{\langle B_k-B_k\rangle}$ is the dense open torus in $X_{B_k}$ part \emph{(\ref{thmb})} follows from the commutativity of the above diagram.

The proof of part \emph{(\ref{thmc})} relies on a series of well know facts for algebraic tori (see \cite{FI14}). Namely since $\langle B_k-B_k\rangle\cap M$ is a sublattice of $M$ one has the following short exact sequence of lattices
\[
\xymatrix{
0 \ar@{->}[r]
&\langle B_k-B_k \rangle_\RR \cap M \ar@{->}[r]
&M \ar@{->}[r]
&M/\langle B_k-B_k\rangle_\RR \cap M \ar@{->}[r]
&0}
\]
The above sequence induces the following short exact sequence on algebraic tori
\[
\xymatrix{
1 \ar@{->}[r]
&T_{M/\langle B_k-B_k\rangle_\RR \cap M}  \ar@{->}[r]
&T_M \ar@{->}[r]^-{g}
&T_{\langle B_k-B_k\rangle_\RR \cap M} \ar@{->}[r]
&1.}
\]
Hence $g^{-1}(1_{T_{\langle B_k-B_k\rangle_\RR \cap M}})=T_{M/\langle B_k-B_k\rangle_\RR \cap M}$ so by  \cite[Lemma 2.1]{FI14} it holds that $\overline{g^{-1}(1_{T_{\langle B_k-B_k \rangle_\RR \cap M}})}$ is projectively equivalent to $X_{\pi(A)}$. If $F$ is an irreducible component of a fiber of $\gamma^k|_{T_M}$ then by \cite[Lemma 2.2]{FI14} $F$ is also a fiber of $g$, i.e. $F$ is a translation of $\overline{g^{-1}(1_{T_{\langle B_k-B_k\rangle_\RR\cap M}})} $ by an element of $T_M$. It now follows from \cite[Lemma 2.1]{FI14} that the closure $\overline{F}$ is projectively equivalent to $X_{\pi(A)}$ proving part \emph{(\ref{thmc})}.
\end{proof}

We now reprove  Theorem~\ref{finite} in the toric setting using a combinatorial approach based on Theorem~\ref{thm:Bklattice}.

\begin{proposition}\label{prop:toricfinite}
Let $X$ be a smooth and projective toric variety  and let $L$ be a $k$-jet spanned line bundle on $X$. Then the general fiber of the Gauss map of order $k$, $\gamma^k$, is finite and birational unless $(X,L)=(\PP^n,\sO(k))$. 
\end{proposition}

\begin{proof}

The pair $(X,L)$ corresponds to a convex lattice polytope $P\subset M_\RR$. Combinatorially the assumption that $X$ is smooth means that the primitive vectors through every vertex of $X$ form a basis for the underlying lattice $M$. Thus we may assume that $P$ is contained in the first orthant and that it has a vertex at the origin, and an edge along each coordinate axis. Moreover, as shown in \cite{DR99},  the assumption that $L$ is $k$-jet spanned corresponds to the fact that every edge of $P$ contains at least $k+1$ lattice points. It follows that $P$ contains the simplex $k\Delta_n=\conv(0,k\hat{e}_1,\ldots,k\hat{e}_n)$, where $\hat{e}_i$ is the unit vector along the $x_i$-axis. There are now two possibilities. The first possibility is that $P=k\Delta_n$, in which case $(X,L)=(\PP^n,\sO(k))$.

If instead $P\supsetneq k\Delta_n$ we consider for every $i\in \{1,\ldots,n\}$ the vertex $v_i$ which lies along the $x_i$-axis and is not the origin. By convexity and because the edges through $v_i$ form a basis for $M$, there is, for every $j\ne i$, an edge through $v_i$ that passes through a point of the form $a\hat{e}_i+\hat{e}_j$ for some $a\in \ZZ$. Because $v_i=b\hat{e}_i$ it then holds that $((a-b)k+b )\hat{e}_i+ k\hat{e}_j \in P$ since every edge of $P$ contains at least $k+1$ lattice points. Thus since $P$ is contained in the first orthant it holds that \[
(a-b)k+b\ge 0 \iff a \ge \frac{b(k-1)}{k} \ge k-1,\]
since $b\ge k$. 
Note that the if $a=k-1$ then convexity implies that $v_i=k\hat{e}_i$ and that the edge through $(k-1)\hat{e}_i+\hat{e}_j$ has its vertices at $v_i=k\hat{e}_i$ and $k\hat{e}_j$. Thus for every $i$ there must be a lattice point in $P$ of the form $k\hat{e}_i+\hat{e}_j$ for some $j$, since the only other possibility is that $v_i=k\hat{e}_i$ and that the edges through $v$ in the $x_ix_j$-plane end in the point $k\hat{e}_j$ for all $j\ne i$. However, with these assumptions, convexity implies that $P=k\Delta_n$ which is a contradiction. Thus for all $i$ there must exist some $j$ such that $k\hat{e}_i+\hat{e}_j\in P$. Note that $S=k\Delta_n\cap M$ is $k$-jet spanned, thus $S_i=(k\Delta_n\cap M)\setminus ((k-1)\hat{e}_i+\hat{e}_j)\cup(k\hat{e}_i+\hat{e}_j)$ is $k$-jet spanned since $(k-1)\hat{e}_i+\hat{e}_j$ is the only lattice point in $k\Delta_n$ which gives a monomial $\mathbf{x}^\mathbf{m}$ such that $\frac{\partial^k}{\partial x_i^{k-1}\partial x_j}(\mathbf{x}^\mathbf{m})$ evaluated at $(1,\ldots,1)$ is non-zero and $\frac{\partial^k}{\partial x_i^{k-1}\partial x_j}(\mathbf{x}^{k\hat{e}_i+\hat{e}_j})(1,\ldots,1)\ne 0$. Set $s=\sum_{u\in S} u$ and $s_i=\sum_{u_i\in S_i} u_i$, then for all $i$, the difference $s_i-s=\hat{e}_i$ lie in $\langle B_k-B_k\rangle$, which implies that $\langle B_k-B_k\rangle_\RR \cap M$ has maximal rank, i.e. the general fiber of $\gamma^k$ is finite, by Theorem~\ref{thm:Bklattice}. Moreover, observe that  the above argument proves that the inclusion  $\langle B_k-B_k\rangle\subseteq M$  is an equality, and hence the induced map of tori is the identity morphism.  This shows that $\gamma^k$ is also birational.  
\end{proof}

\begin{figure}
\begin{minipage}[t]{0.43\linewidth}
\centering
\begin{tikzpicture}[scale=0.7]
\draw[thick,->] (-1,0)--(5,0);
\draw[thick,->] (0,-1)--(0,5);
\draw[ultra thick,fill=gray,fill opacity=0.3] (0,0)--(2,0)--(4,2)--(4,4)--(2,4)--(0,2)--(0,0);
\draw[ultra thick, dashed] (2,0)--(0,2);
\foreach \x in {0,1,2,3,4}{
\foreach \y in {0,1,2,3,4}{
\node[fill=black, shape=circle, scale=0.5] at (\x,\y) {};
}
}
\node at (2,-0.5) {$v_1$};
\node at (-0.5,2) {$v_2$};
\node at (3.9,0.5) {$a\hat{e}_1+\hat{e}_2$};
\node at (0.7,0.5) {$k\Delta_2$};
\end{tikzpicture}

\captionsetup{width=0.72\textwidth}
\caption{Illustration for the proof of Proposition~\ref{prop:toricfinite}.}

\end{minipage}
\begin{minipage}[t]{0.56\linewidth}
\centering
\begin{tikzpicture}
\draw (0.866025,1,0.5)--(0,1,0)--(0.5,3,-0.866025)--(1,0,-1.73205)--(1.73205,0, 1)--(0.866025,1,0.5);
\draw (0.5, 3, -0.866025)--(1.73205, 0, 1);
\draw[dashed] (0,0,0)--(1, 0, -1.73205);
\draw (0,0,0)--(1.73205, 0, 1);
\draw (0,0,0)--(0,1,0);
\draw (0.866025,1,0.5)--(0.5,3,-0.866025);
\foreach \pt in {(0,0,0),(0,1,0),(0.866025,0,0.5),(0.866025,1,0.5),(1,0,-1.73205),(1.73205,0, 1),(1.36603, 0., -0.366025)}{
\node [fill=blue,shape=circle, scale=0.5] at \pt {};
}

\foreach \n in {0,1}{
\foreach \pt in {(0.5, \n., -0.866025)}{
\node [fill=green,shape=circle, scale=0.5] at \pt {};
}}

\foreach \n in {2,3}{
\foreach \pt in {(0.5, \n., -0.866025)}{
\node [fill=red,shape=circle, scale=0.5] at \pt {};
}}

\end{tikzpicture}
\captionsetup{width=0.87\textwidth}
\caption{Illustration for Example~\ref{ex:infinitefibers}, with $n=N=3$. The set $S_3$ consists of the blue and green points, while the set $L_3$ consists of the green and red points.}
\end{minipage}
\end{figure}

Observe that if $L$ is a generically $2$-jet spanned line bundle on a $n$-dimensional toric variety $X$, then $\dim(H^0(X,L))\ge \binom{n+2}{2}$. Below we give an example of a generically $2$-jet spanned pair $(X,L)$, with $X$ singular, such that the Gauss map of order $2$ has infinite fibers, $\dim(X)=n$ and $\dim(H^0(X,L))=\binom{n+2}{2}+N-2$ for all $n,N\ge 2$. These examples were found using the package \cite{LP} for \textsf{Macaulay2} which uses Theorem~\ref{thm:Bklattice}  to compute the image and fiber of the Gauss map of order $k$ in the toric setting.

\begin{example}\label{ex:infinitefibers}
For every pair of integers $n\ge 2$ and $N\ge 2$ we define the convex lattice polytope $P_n^N=\conv(A_1,A_2,\ldots A_n)\subset M_\RR=M\otimes \RR$ where
\[
A_1=\{0,\hat{e}_1+N\hat{e}_2,2\hat{e}_1\}, A_2=\{\hat{e}_2\}, A_j=\{\hat{e}_1+\hat{e}_j,2\hat{e}_j\} \text{ for }2<j\le n
\]
and $\hat{e}_1,\ldots,\hat{e}_n$ is a basis for $M$. We claim that the Gauss map of order $2$ for the projective and normal variety $X_{P_n^N\cap M}$ and generically $2$-jet spanned line bundle $L$ associated to $P_n^N$ has positive dimensional fibers. For every $n$ let $S_n$ be the set of lattice points corresponding to monomials of degree at most $2$ in the variables $x_1,\ldots, x_n$, but without the lattice point $2\hat{e}_2$. Moreover for every $N\ge 2$ let $L_N$ be the set of lattice points of the form $\hat{e}_1+m\hat{e}_2$ for $0\le m\le N$. Then by considering the fibers under the projection onto the $x_1x_2$-plane one readily checks that the lattice points of $P_n^N$ decompose as $P_n^N\cap M=S_n \cup L_N$.

By the above we have that $P_n^2\cap M\subseteq P_n^N\cap M$ for all $n\ge 2$ and $N\ge 2$. By direct computation one checks that $P_n^2$ is generically $2$-jet spanned. Thus $P_n^N$ is generically $2$-jet spanned for all $N\ge 2$. Now for any $P_n^N$ consider the column-span $C$ of the columns in the matrix of $2$-jets corresponding to all lattice points in $L_N$. Using linear algebra techniques one readily checks that $\dim(C)=3$. Thus every generically $2$-jet spanned subset $A$ of $P_n^N\cap M$ such that $|A|=\binom{n+2}{2}$ can contain at most $3$ lattice points in $L_N$. However, as there are exactly $\binom{n+2}{2}-1-2$ lattice points in $S_n\setminus L_N$, it then follows that every such subset $A$ of $P_n^N\cap M=S_n\cup L_N$ is determined by the choice of three lattice points in $L_N$. Now, by definition, every element of $B_k$ is the sum of the elements of a generically $2$-jet spanned subset $A\subseteq P_n^N\cap M$ such that $|A|=\binom{n+2}{2}$. As the only difference between two such subsets lies in the choice of $3$ lattice points in $L_N$, the only difference between two elements in $B_k$ is their $x_2$-coordinate. Thus $\langle B_k-B_k\rangle_\RR$ has dimension $0$ if $|P_n^N\cap M|=\binom{n+2}{2}$ and dimension $1$ otherwise. As a consequence the fibers of the Gauss map of order $2$ for the projective, normal and singular toric $n$-fold $X_{P_n^N\cap M}$ have dimension $n$ if $|P_n^N\cap M|=\binom{n+2}{2}$ and dimension $n-1$ if $|P_n^N\cap M|>\binom{n+2}{2}$ by Theorem~\ref{thm:Bklattice}.

\end{example}

\begin{remark}
In  \cite[Corollary 1.3]{FI14} the authors show that the sets of lattice points $A$ giving degenerate Gauss maps are so called Cayley sums. Example~\ref{ex:infinitefibers} shows that this characterization does not directly generalize to higher order since the sets of lattice points appearing there are not Cayley sums. We leave it as an open problem to characterize sets of lattice points yielding degenerate Gauss maps of order $k$.
\end{remark}


\end{document}